\documentclass[14pt]{amsart}

\usepackage{cases}
\usepackage{amsmath}
\usepackage{amsfonts}
\usepackage{bm}
\usepackage{amsfonts,amsmath,amssymb,amscd,bbm,amsthm,mathrsfs,dsfont}
\usepackage{mathrsfs}
\usepackage{pb-diagram}
\usepackage{amssymb}
\usepackage{xypic}

\newtheorem{Theorem}{Theorem}[section]
\newtheorem{Lemma}[Theorem]{Lemma}

\newtheorem{Definition}[Theorem]{Definition}

\newtheorem{Proposition}[Theorem]{Proposition}

\newtheorem{Conjecture}[Theorem]{Conjecture}

\date{version of \today}

 \normalbaselines
\title
[Positivity of denominator vectors of cluster algebras]
{Positivity of denominator vectors of cluster algebras}

\author{Peigen Cao $\;\;\;\;\;\;$ Fang Li $\;\;\;\;\;\;$}
\address{Peigen Cao
\newline Department
of Mathematics, Zhejiang University (Yuquan Campus), Hangzhou, Zhejiang
310027,  P.R.China}
\email{peigencao@126.com}

\address{Fang Li
\newline Department
of Mathematics, Zhejiang University (Yuquan Campus), Hangzhou, Zhejiang
310027, P.R.China}
\email{fangli@zju.edu.cn}

\begin{document}
\renewcommand{\thefootnote}{\alph{footnote}}

\renewcommand{\thefootnote}{\alph{footnote}}
\setcounter{footnote}{-1} 
\renewcommand{\thefootnote}{\alph{footnote}}
\setcounter{footnote}{-1} 

\begin{abstract}
In this paper, we prove that positivity of denominator vectors holds for any skew-symmetric cluster algebra.
\end{abstract}

\maketitle
\bigskip

\section{introduction}

Cluster algebras were introduced by Fomin and Zelevinsky in \cite{FZ}. The motivation was to create a common framework for phenomena occurring in connection with total positivity and canonical bases. A cluster algebra $\mathcal A(\mathcal S)$ of rank $n$ is a subalgebra of an ambient field $\mathcal F$ generated by certain combinatorially defined generators (i.e.,  {\em cluster variables}), which are grouped into overlapping {\em clusters}. One of the important features of cluster algebras is their Laurent phenomenon. Thus  the {\em denominator vectors} of cluster variables can be defined  following this phenomenon in Section 2.  Fomin and Zelevinsky conjectured that
\begin{Conjecture}\label{conjd}
{\em (Conjecture 7.4 (i) in \cite{FZ3})}\; Positivity of denominator vectors) For any given a cluster ${\bf x}_{t_0}$ of a cluster algebra  $\mathcal A(\mathcal S)$, and any cluster variable  $x\notin {\bf x}_{t_0}$,  the denominator vector  $d^{t_0}(x)$ of  $x$ with respect to ${\bf x}_{t_0}$ is in $\mathbb N^n$.
\end{Conjecture}

 In the sequel, we always write denominator vectors briefly as $d$-vectors.

Positivity of $d$-vectors has been affirmed in the following cases:
 (i) cluster algebras of rank $2$ (Theorem 6.1 of \cite{FZ}), (ii) cluster algebras arising from surfaces (see \cite{FST}), and (iii)  cluster algebras of finite type (see \cite{CCS,CP}).

  Caldero and  Keller proved a weak version of Conjecture \ref{conjd} in \cite{CK}, that is, for a skew-symmetric cluster algebra  $\mathcal A(\mathcal S)$, if its seed $({\bf x}_{t_0},{\bf y}_{t_0},B_{t_0})$ is acyclic, then the  $d$-vector $d^{t_0}(x)$ of any cluster variable $x\notin {\bf x}_{t_0}$ is in $\mathbb N^n$.

In this paper, we will give a positive answer to Conjecture \ref{conjd} for skew-symmetric cluster algebras.

\begin{Theorem}\label{mainthm}
Positivity of $d$-vectors holds for any skew-symmetric cluster algebra.
\end{Theorem}

This paper is organized as follows. In Section 2, some basic definitions and notations are introduced. In Section 3, the proof of Theorem \ref{mainthm} is given.

\section{Preliminaries}

Recall that $(\mathbb P, \oplus, \cdot)$ is a {\bf semifield } if $(\mathbb P,  \cdot)$ is an abelian multiplicative group endowed with a binary operation of auxiliary addition $\oplus$ which is commutative, associative, and distributive with respect to the multiplication $\cdot$ in $\mathbb P$.

Let $Trop(u_i:i\in I)$ be a free abelian group generated by $\{u_i: i\in I\}$ for a finite set of index $I$. We define the addition $\oplus$ in $Trop(u_i:i\in I)$ by $\prod\limits_{i}u_i^{a_i}\oplus\prod\limits_{i}u_i^{b_i}=\prod\limits_{i}u_i^{min(a_i,b_i)}$, then $(Trop(u_i:i\in I), \oplus)$ is a semifield, which is called a {\bf tropical semifield}.

The multiplicative group of any semifield $\mathbb P$ is torsion-free for multiplication \cite{FZ}, hence its group ring $\mathbb Z\mathbb P$ is a domain.
We take an ambient field $\mathcal F$  to be the field of rational functions in $n$ independent variables with coefficients in $\mathbb Z\mathbb P$.

 An $n\times n$ integer matrix $B=(b_{ij})$ is called {\bf skew-symmetric } if $b_{ij}=-b_{ji}$ for any $i,j=1,\cdots,n$; and in general,  $B$ is called  {\bf skew-symmetrizable} if there exists a positive integer diagonal matrix $S$  such that $SB$ is skew-symmetric.

  For  a skew-symmetrizable matrix $B$, we can encode the sign pattern of matrix entries of $B$ by the directed graph $\Gamma(B)$ with the vertices $1,2,\cdots,n$ and the directed edges $(i,j)$ for $b_{ij}>0$. A skew-symmetrizable matrix matrix $B$ is called {\bf acyclic} if $\Gamma(B)$ has no oriented cycles.

\begin{Definition}
A {\bf seed}  in $\mathcal F$ is a triplet $({\bf x},{\bf y},B)$ such that

(i)  ${\bf x}=\{x_1,\cdots, x_n\}$,  called a {\bf cluster}, is a free generating set of $\mathcal F$, where $x_1\cdots,x_n$  are called {\bf cluster variables}.

(ii) ${\bf y}=\{y_1,\cdots,y_n\}$ is a subset of $\mathbb P$,  where  $y_1,\cdots,y_n$ are called {\bf coefficients}.

(iii) $B=(b_{ij})$ is a $n\times n$ skew-symmetrizable  integer matrix, called an {\bf exchange matrix}.

The seed $({\bf x},{\bf y},B)$ is called {\bf acyclic} if $B$ is acyclic.
\end{Definition}

Denote by $[a]_+=max\{a,0\}$ for any $a\in \mathbb R$.
\begin{Definition}
Let $({\bf x},{\bf y},B)$ be a seed in $\mathcal F$. Define the {\bf mutation}  of the seed $({\bf x},{\bf y},B)$  at $k\in\{1,\cdots,n\}$ as a new triple $\mu_k({\bf x},{\bf y},B)=(\bar{\bf x},\bar{\bf y},\bar B)$ given by
\begin{eqnarray}
\label{eq1}\bar x_i&=&\begin{cases}x_i~,&\text{if } i\neq k\\ \frac{y_k\prod\limits_{j=1}^nx_j^{[b_{jk}]_+}+ \prod\limits_{j=1}^nx_j^{[-b_{jk}]_+}}{({1\bigoplus y_k})x_k},~& \text{if }i=k.\end{cases}\\
\bar y_i&=&\begin{cases} y_k^{-1}~,& i=k\\ y_iy_k^{[b_{ki}]_+}(1\bigoplus y_k)^{-b_{ki}}~,&otherwise.
\end{cases}\\
\bar b_{ij}&=&\begin{cases}-b_{ij}~,& i=k\text{ or } j=k;\\ b_{ij}+b_{ik}[-b_{kj}]_++[b_{ik}]_+b_{kj}~,&otherwise.\end{cases}.
\end{eqnarray}

\end{Definition}
It can be seen that $\mu_k({\bf x},{\bf y},B)$ is also a seed and $\mu_k(\mu_k({\bf x},{\bf y},B))=({\bf x},{\bf y},B)$.

\begin{Definition}
  A {\bf cluster pattern}  $\mathcal S$ in $\mathcal F$ is an assignment of a seed  $({\bf x}_t,{\bf y}_t,B_t)$ to every vertex $t$ of the $n$-regular tree $\mathbb T_n$, such that for any edge $t^{~\underline{\quad k \quad}}~ t^{\prime},~({\bf x}_{t^\prime},{\bf y}_{t^\prime},B_{t^\prime})=\mu_k({\bf x}_t,{\bf y}_t,B_t)$.
\end{Definition}
We always denote by ${\bf x}_t=\{x_{1;t},\cdots, x_{n;t}\},~ {\bf y}_t=\{y_{1;t},\cdots, y_{n;t}\}, ~B_t=(b_{ij}^t).$
\begin{Definition} Let  $\mathcal S$ be a cluster pattern,   the {\bf cluster algebra} $\mathcal A(\mathcal S)$  associated with the given cluster pattern $\mathcal S$ is the $\mathbb {ZP}$-subalgebra of the field $\mathcal F$ generated by all cluster variables of  $\mathcal S$.
 \end{Definition}
 \begin{itemize}
 \item If  $\mathcal S$ is cluster pattern with coefficients in $Trop(y_1,\cdots,y_m)$, the corresponding cluster algebra $\mathcal A(\mathcal S)$ is said to be a cluster algebra of {\bf geometric type}.
 \item If  $\mathcal S$ is cluster pattern with coefficients in $Trop(y_1,\cdots,y_n)$ and there exists a seed $({\bf x}_{t_0},{\bf y}_{t_0},B_{t_0})$ such that $y_{i;t_0}=y_i$ for $i=1,\cdots,n$, then  the corresponding cluster algebra  $\mathcal A(\mathcal S)$ is called a {\bf cluster algebra with principal coefficients at $t_0$}.
\end{itemize}

\begin{Theorem} \label{laurent} Let $\mathcal A(\mathcal S)$ be a skew-symmetrizable cluster algebra, and $({\bf x}_{t_0},{\bf y}_{t_0},B_{t_0})$ be a  seed of $\mathcal A(\mathcal S)$. Then

(i).(Theorem 3.1 in \cite{FZ}, Laurent phenomenon)
 any cluster variable $x_{i;t}$ of $\mathcal A(\mathcal S)$ is a sum of Laurent monomials in ${\bf x}_{t_0}$ with coefficients in $\mathbb {ZP}$.

(ii). (\cite{GHKK}, positive  Laurent phenomenon)  any cluster variable $x_{i;t}$ of $\mathcal A(\mathcal S)$ is a sum of Laurent monomials in ${\bf x}_{t_0}$ with coefficients in $\mathbb {NP}$.
\end{Theorem}

Denote by ${\bf x}_{t_0}^{\bf a}:=\prod\limits_{i=1}^nx_{i;t_0}^{a_i}$ for ${\bf a}\in \mathbb Z^n$. Let $\mathcal A(\mathcal S)$ be a skew-symmetrizable cluster algebra, and $({\bf x}_{t_0}, {\bf y}_{t_0}, B_{t_0})$ be a seed of  $\mathcal A(\mathcal S)$, by Laurent phenomenon,  any cluster variable $x$ of $\mathcal A(\mathcal S)$ has the form of $x=\sum\limits_{{\bf v}\in V} c_{\bf v}{\bf x}_{t_0}^{\bf v}$, where $V$ is a subset of $\mathbb Z^n$, $0\neq c_{\bf v}\in \mathbb {ZP}$. Let $-d_{j}$ be the minimal exponent of $x_{j;t_0}$ appearing in the expansion $x=\sum\limits_{{\bf v}\in V} c_{\bf v}{\bf x}_{t_0}^{\bf v}$. Then $x$ has the form of
\begin{eqnarray}
\label{eqd}x=\frac{f(x_{1;t_0},\cdots,x_{n,t_0})}{x_{1;t_0}^{d_1}\cdots x_{n;t_0}^{d_n}},
\end{eqnarray}
where $f\in\mathbb {ZP}[x_{1;t_0},\cdots,x_{n;t_0}]$ with  $x_{j;t_0}\nmid f$ for $j=1,\cdots,n$.
The vector $d^{t_0}(x) = (d_1,\cdots, d_n)^{\top}$ is called the {\bf denominator vector} (briefly, {\bf $d$-vector})  of the cluster variable $x$ with respect to ${\bf x}_{t_0}$.

\begin{Proposition}{\em ((7.7) in \cite{FZ3})}\label{dmutation}
 Denote the standard basis vectors in $\mathbb Z^n$ by ${\bf e}_1,\cdots,{\bf e}_n$. The vectors $d^{t_0}(x_{l;t})$ are uniquely determined by the initial conditions $$d^{t_0}(x_{l;t_0})=-{\bf e}_l$$
together with the recurrence relations

$$d^{t_0}(x_{l;t^\prime})=\begin{cases} d^{t_0}(x_{l;t})&\text{ if } l\neq k;\\
-d^{t_0}(x_{k;t})+max\left(\sum\limits_{b_{ik}^t> 0}b_{ik}^td^{t_0}(x_{i;t}),\sum\limits_{b_{ik}^t<0}-b_{ik}^td^{t_0}(x_{i;t})\right)& \text{ if } l=k \end{cases}$$
for  $t^{~\underline{\quad k \quad}}~ t^{\prime}$.
\end{Proposition}

From the above proposition, we know that the notion of $d$-vectors is independent of the choice of coefficient system. So when studying the $d$-vector $d^{t_0}(x)$ of a cluster variable $x$, we can focus on the cluster algebras of geometric type.

\begin{Proposition}\label{prodvectors}
{\em (Proposition 2.5 in \cite{RS})}
 Let $\mathcal A(\mathcal S)$ be a cluster algebra, $x$ be a cluster variable, and $({\bf x}_t,{\bf y}_t,B_t)$, $({\bf x}_{t^\prime}, {\bf y}_{t^\prime},B_{t^\prime})$ be two seeds of  $\mathcal A(\mathcal S)$
with $({\bf x}_{t^\prime}, {\bf y}_{t^\prime},B_{t^\prime})=\mu_k({\bf x}_t,{\bf y}_t,B_t)$. Suppose that $d^t(x)=(d_1,\cdots,d_n)^{\top}$, $d^{t^\prime}(x)=(d_1^\prime,\cdots,d_n^\prime)$ are the $d$-vectors of $x$ with respective to ${\bf x}_t$ and ${\bf x}_{t^\prime}$ respectively, then $d_i=d_i^{\prime}$ for $i\neq k$.
\end{Proposition}
\begin{proof}
For the convenience of the readers, we give the sketch of the proof here. By  $({\bf x}_{t^\prime}, {\bf y}_{t^\prime},B_{t^\prime})=\mu_k({\bf x}_t,{\bf y}_t,B_t)$, we know that $x_{k;t}x_{k;t^\prime}=\frac{y_{k;t}}{1\oplus y_{k;t}}\prod\limits_{b_{ik}^{t}>0}x_{i;t}^{b_{ik}^{t}}+\frac{y_{k;t}}{1\oplus y_{k;t}}\prod\limits_{b_{ik}^{t}<0}x_{i;t}^{-b_{ik}^{t}}.$ Thus

\begin{eqnarray} \label{eqn2}
x_{k;t}=x_{k;t^\prime}^{-1}(\frac{y_{k;t}}{1\oplus y_{k;t}}\prod\limits_{b_{ik}^{t}>0}x_{i;t}^{b_{ik}^{t}}+\frac{y_{k;t}}{1\oplus y_{k;t}}\prod\limits_{b_{ik}^{t}<0}x_{i;t}^{-b_{ik}^{t}}).
\end{eqnarray}

By Laurent phenomenon, the expansion of $x$ with respect to ${\bf x}_t$ has the following form

\begin{eqnarray} \label{eqn3}
x=\frac{f(x_{1;t},\cdots,x_{n,t})}{x_{1;t}^{d_1}\cdots x_{n;t}^{d_n}},
\end{eqnarray}
where $f\in\mathbb {ZP}[x_{1;t},\cdots,x_{n;t}]$ with  $x_{j;t}\nmid f$ for $j=1,\cdots,n$.
 Then replacing the $x_{k;t}$ in the equation (\ref{eqn3}) by the one in the equation (\ref{eqn2}), we get the expansion of $x$ with respect to ${\bf x}_{t^\prime}$. And it can be seen that $d_i=d_i^{\prime}$ for $i\neq k$.
\end{proof}

\section{ The proof of Theorem \ref{mainthm}}

Let $\mathcal A(\mathcal S)$ be a cluster algebra and $z$ be a cluster variable of $\mathcal A(\mathcal S)$. Denote by $I(z)$  the set of clusters ${\bf x}_{t_0}$ of $\mathcal A(\mathcal S)$ such that $z$ is a cluster variable in ${\bf x}_{t_0}$. For two vertices $t_1,t_2$ of $\mathbb T_n$, let $l(t_1,t_2)$ be the distance between $t_1$ and $t_2$ in the $n$-regular tree $\mathbb T_n$.  For a cluster ${\bf x}_t$ of $\mathcal A(\mathcal S)$, define the distance  between $z$ and ${\bf x}_t$ as $dist(z,{\bf x}_t):=min\{l(t_0,t)|\; t_0\in I(z)\}$.

The following theorem is from \cite{LS}:
 \begin{Theorem}{\em (Theorem 4.1, Proposition 5.1 and 5.5 in \cite{LS})\label{Lee}}

 Let $\mathcal A(\mathcal S)$ be a skew-symmetric cluster algebra of geometric
type with a cluster ${\bf x}_t$. Let $z$ be a cluster variable of $\mathcal A(\mathcal S)$, and ${\bf x}_{t_0}$ be a cluster containing $z$ such that $dist(z,{\bf x}_t)=l(t_0,t)$. Let $\sigma$ be the unique
sequence in $\mathbb T_n$ relating the seeds at $t_0$ and $t$ in which $p,q$ denote the last two directions (clearly, $p\neq q$):
$$\sigma:\;\; t_0^{~\underline{  \quad k_1\quad   }}~ t_1^{~\underline{\quad k_2 \quad}}  ~\cdots ~t_{m-2}^{~\underline{\quad p=k_{m-1} \quad}}~t_{m-1}=u ^{~\underline{~\quad q=k_{m} \quad}}~ t_m=t.$$
For $e\neq p,q$, let $u^{~\underline{  \quad e \quad   }}~ v$ and $t^{~\underline{  \quad e \quad   }}~ w$ be the edges of $\mathbb T_n$ in the same direction $e$. Then,

(i) $z$ can be written as $z=P_t+Q_t$, with $P_t\in L_1:=\mathbb {NP}[x_{q;u},x_{p;t}^{\pm1};(x_{i;t}^{\pm1})_{i\neq p,q}]$, and $Q_t\in L_2:=\mathbb {NP}[x_{q;t},x_{p;t}^{\pm1};(x_{i;t}^{\pm1})_{i\neq p,q}]$.  Moreover, $P_t$ and $Q_t$ are unique up to $L_1\cap L_2$.

(ii)  There exists a Laurent monomial $F$ appearing in the expansion $z=P_t+Q_t$  such that the variable $x_{e;t},$  has nonnegative exponent in $F$.

(iii)   There exist $P_{1;t}\in L_3,P_{2;t}\in L_4,Q_{1;t}\in L_5,Q_{2;t}\in L_6$ such that  $P_{1;t}+P_{2;t}=P_t$ and $Q_{1;t}+Q_{2;t}=Q_t$, where
\begin{eqnarray}
 L_3&:=&\mathbb {NP}[x_{q;u},x_{e;v};(x_{i;t}^{\pm1})_{i\neq p,q}],~~~~L_4:=\mathbb {NP}[x_{q;u},x_{e;u};(x_{i;t}^{\pm1})_{i\neq p,q}],\nonumber\\
L_5&:=&\mathbb {NP}[x_{q;t},x_{e;t};(x_{i;t}^{\pm1})_{i\neq p,q}],~~~~L_6:=\mathbb {NP}[x_{q;t},x_{e;w};(x_{i;t}^{\pm1})_{i\neq p,q}].\nonumber
\end{eqnarray}
Thus $z$ has the form of $z=P_{1;t}+P_{2;t}+Q_{1;t}+Q_{2;t}$, where $P_{1;t},P_{2;t},Q_{1;t},Q_{2;t}$ are unique up to $L_3\cap L_4\cap L_5\cap L_6$.
 \end{Theorem}
\begin{Lemma}\label{lempq} Keep the notations in Theorem \ref{Lee}.
There exists a Laurent monomial in Laurent expansion of $z$ with respect ${\bf x}_t$ in which  $x_{e;t}$ appears with nonnegative exponent for $e\neq p,q$.
\end{Lemma}
\begin{proof}
By Theorem \ref{Lee} (ii), there exists a Laurent monomial $F$ appearing in the expansion $z=P_t+Q_t$  such that the variable $x_{e;t}$  has non-negative exponent in $F$.
We know $$x_{q;u}=\frac{y_{q;t}\prod\limits_{j=1}^nx_{j;t}^{[b_{jq}^t]_+}+ \prod\limits_{j=1}^nx_{j;t}^{[-b_{jk}^t]_+}}{({1\bigoplus y_{q;t}})x_{q;t}}.$$
Substituting the above equality into $z=P_t+Q_t$, then we obtain the Laurent expansion of $z$ with respect to ${\bf x}_t$. It is easy to see that in this Laurent expansion there exists a Laurent monomial such that  $x_{e;t}$ appear in this Laurent monomial with nonnegative exponent.
\end{proof}

{\em Now we give  the proof of Theorem \ref{mainthm}.}

\begin{proof}
By Proposition \ref{dmutation}, we know the notion of $d$-vectors is independent of the choice of coefficient system. Hence, we can assume that  $\mathcal A(\mathcal S)$ is a skew-symmetric cluster algebra of geometric type.

Let ${\bf x}_{t^{\prime}}$ be a cluster of $\mathcal A(\mathcal S)$, $z$ be a cluster
variable with $z\notin {\bf x}_{t^{\prime}}$, then $dist(z,{\bf x}_{t^{\prime}})=:m+1>0$. We will show the $d$-vector $d^{t^{\prime}}(z)$ of $z$ with respect to ${\bf x}_{t^{\prime}}$ is in $\mathbb N^n$.

 Let ${\bf x}_{t_0}\in I(z)$ such that $l(t_0,t^{\prime})=dist(z,{\bf x}_{t^{\prime}})=m+1>0$.  Since $\mathbb T_n$ is a tree, there is  an unique sequence $\sigma$ linking $t_0$ and $t^{\prime}$ in $\mathbb T_n$:
$$t_0^{~\underline{  \quad k_1\quad   }}~ t_1^{~\underline{\quad k_2 \quad}}  ~\cdots ~t_{m-2}^{~\underline{\quad p=k_{m-1} \quad}}~t_{m-1}=u ^{~\underline{~\quad q=k_{m} \quad}}~ t_m=t~^{~\underline{\quad k_{m+1}(\neq q) \quad}}~t_{m+1}=t^{\prime}.$$
By the choice of ${\bf x}_{t_0}$, we know $dist(z,{\bf x}_{t_j})=l(t_0,t_j)=j$ for $j=1,2,\cdots,m+1$.

We prove  $d^{t^{\prime}}(z)\in\mathbb N^n$ by induction on $dist(z,{\bf x}_{t^{\prime}})=m+1>0$.

 Clearly,  $d^{t^{\prime}}(z)\in\mathbb N^n$ if $dist(z,{\bf x}_{t^{\prime}})=1,2$.

 {\em Inductive Assumption}:  Assume that $d^{t^{\star}}(z)\in\mathbb N^n$ for any $t^{\star}$ such that $1\leq dist(z,{\bf x}_{t^{\star}})<dist(z, {\bf x}_{t^{\prime}})$.

Because $1\leq dist(z,{\bf x}_t)<dist(z, {\bf x}_{t^{\prime}})$, we know $d^{t}(z)\in\mathbb N^n$, by inductive assumption. By Proposition \ref{prodvectors} and ${\bf x}_{t^{\prime}}=\mu_{k_{m+1}}({\bf x}_t)$, we get that the $i$-th component of $d^{t}(z)$ and the $i$-th component of $d^{t^{\prime}}(z)$ are equal for $i\neq k_{m+1}$. Thus, in order to show $d^{t^{\prime}}(z)\in\mathbb N^n$, we only need to show that the $k_{m+1}$-th component of $d^{t^{\prime}}(z)$ is nonnegative.
We prove this in two cases. Case (I): $k_{m+1}\neq p$ and Case (II): $k_{m+1}=p$.

The proof in Case (I):

By positive Laurent phenomenon, the expansion of $z$ with respect to ${\bf x}_{t^{\prime}}$ has the form of $z=\sum\limits_{{\bf v}\in V}c_{\bf v}{\bf x}_{t^{\prime}}^{\bf v}$, where $V$ is a subset of $\mathbb Z^n$, and $0\neq c_{\bf v}\in \mathbb {NP}$. If $k_{m+1}$-th component of $d^{t^{\prime}}(z)$ is negative (here $k_{m+1}\neq p,q$), then  the exponent of $x_{k_{m+1};t^{\prime}}$ in each ${\bf x}_{t^{\prime}}^{\bf v}$ must be positive, ${\bf v}\in V$. We know $$x_{k_{m+1};t^{\prime}}=\frac{y_{k_{m+1};t^{\prime}}\prod\limits_{j=1}^nx_{j;t^{\prime}}^{[b_{jk_{m+1}}^{t^{\prime}}]_+}+ \prod\limits_{j=1}^nx_{j;t^{\prime}}^{[-b_{jk_{m+1}}^{t^{\prime}}]_+}}{({1\bigoplus y_{k_{m+1};t^{\prime}}})x_{k_{m+1};t}}.$$
Substituting the above equation into $z=\sum\limits_{{\bf v}\in V}c_{\bf v}{\bf x}_{t^{\prime}}^{\bf v}$, we can obtain the Laurent expansion of $z$ with respect to ${\bf x}_t$. And the exponent of $x_{k_{m+1};t}$ in each Laurent monomial appearing in the obtained  Laurent expansion is negative. This contradicts Lemma \ref{lempq}.
Thus if $k_{m+1}\neq p,q$, then $k_{m+1}$-th component of $d^{t^{\prime}}(z)$ is nonnegative and we have  $d^{t^{\prime}}(z)$ is in $\mathbb N^n$.

The preparation for the proof in Case (II):

Consider the maximal rank two mutation subsequence in directions $p,q$ at the end of $\sigma$. This sequence connects $t^{\prime}$ to a vertex $t_r$.  Thus we have the following two casesㄩ
\begin{eqnarray}\label{eqnseq1}
t_0\cdots \cdot~t_{r-1}^{~\underline{  \quad k_r(\neq q)\quad   }}~ t_r^{~\underline{\quad p \quad}}~t_{r+1}^{~\underline{\quad q \quad}}~t_{r+2}^{~\underline{\quad p \quad}} ~\cdots ~t_{m-2}^{~\underline{\quad p \quad}}~t_{m-1}^{~\underline{\quad q \quad}}~t_m^{~\underline{~\quad p \quad}}~t_{m+1}=t^{\prime}.
\end{eqnarray}
 or
\begin{eqnarray}\label{eqnseq2}
t_0\cdots \cdot~t_{r-1}^{~\underline{  \quad k_r(\neq p)\quad   }}~ t_r^{~\underline{\quad q \quad}}~t_{r+1}^{~\underline{\quad p \quad}}~t_{r+2}^{~\underline{\quad q \quad}}~\cdots ~t_{m-2}^{~\underline{\quad p \quad}}~t_{m-1}^{~\underline{\quad q \quad}}~t_m^{~\underline{~\quad p \quad}}~t_{m+1}=t^{\prime}.
\end{eqnarray}

Let $t_r^{~\underline{\quad q \quad}}~t_{a}$  in the first case and let $t_r^{~\underline{\quad p \quad}}~t_{b}$  in the second case. It is easy to see that $r\leq m-2$ in the first case and $r\leq m-3$ in the second case.

\begin{Lemma}\label{lempositive}
(i) In the first case $t_r^{~\underline{\quad q \quad}}~t_{a}$,  the $p$-th components of $d$-vectors of $x_{p;t_a},x_{q;t_a},x_{p;t_r},x_{p;t_{r+1}}$, $x_{q;t_{r+1}},x_{q;t_{r+2}}$ with respect to ${\bf x}_{t^{\prime}}$
are nonnegative.

(ii) In the second case $t_r^{~\underline{\quad p \quad}}~t_{b}$,  the $p$-th components of $d$-vectors of $x_{p;t_b},x_{q;t_b},x_{q;t_r},x_{p;t_{r+1}},$ $x_{q;t_{r+1}}, x_{p;t_{r+2}}$ with respect to ${\bf x}_{t^{\prime}}$
are nonnegative.
\end{Lemma}
\begin{proof}
(i) If the $p$-th components of $d$-vector $d^{t^{\prime}}(x)$ of some cluster variable $x\in\{x_{p;t_a},x_{q;t_a},x_{p;t_r},$ $x_{p;t_{r+1}},x_{q;t_{r+1}},x_{q;t_{r+2}}\}$  with respect to ${\bf x}_{t^{\prime}}$ is negative. Say $x\in {\bf x}_{t_E}, E\in\{a,r,r+1,r+2\}$.
Since the clusters  ${\bf x}_{t_a},{\bf x}_{t_r},{\bf x}_{t_{r+1}}$, ${\bf x}_{t_{r+2}}$  can be obtained from the cluster ${\bf x}_{t^{\prime}}$  by  sequences of mutations using only  two directions $p$ and $q$. Then by Theorem 6.1 in \cite{FZ}, we know $x=x_{p;t^{\prime}}$. Hence, ${\bf x}_{t_E}$ and ${\bf x}_{t^{\prime}}$ have at least $n-1$ common cluster variables. By Theorem 5 of \cite{GSV} or Theorem 4.23 (c) of \cite{CL}, we know ${\bf x}_{t^{\prime}}={\bf x}_{t_E}$ or ${\bf x}_{t^{\prime}}=\mu_p({\bf x}_{t_E})$ or ${\bf x}_{t^{\prime}}=\mu_q({\bf x}_{t_E})$.
Thus
\begin{eqnarray}\label{eqnless}
m+1=dist(z,{\bf x}_{t^{\prime}})\leq max\{dist(z,{\bf x}_{t_E}),dist(z,\mu_p({\bf x}_{t_E}),dist(z,\mu_q({\bf x}_{t_E})\}.\nonumber
\end{eqnarray}

 On the other hand, by the sequence (\ref{eqnseq1}) and $t_r^{~\underline{\quad q \quad}}~t_{a}$, we know $$max\{dist(z,{\bf x}_{t_E}),dist(z,\mu_p({\bf x}_{t_E}),dist(z,\mu_q({\bf x}_{t_E})\}\leq r+3\leq m+1.$$
Thus we obtain $max\{dist(z,{\bf x}_{t_E}),dist(z,\mu_p({\bf x}_{t_E}),dist(z,\mu_q({\bf x}_{t_E})\}=m+1$, and this will result in that $r=m-2$ and $E=m$, i.e. ${\bf x}_{t_E}={\bf x}_{t_m}$. But this  contradicts that $x_{p;t^{\prime}}=x\in {\bf x}_{t_E}={\bf x}_{t_m}$. So  the $p$-th components of $d$-vectors of $x_{p;t_a},x_{q;t_a},x_{p;t_r},x_{p;t_{r+1}}$, $x_{q;t_{r+1}},x_{q;t_{r+2}}$ with respect to ${\bf x}_t^{\prime}$ are nonnegative.

(ii) By the same argument with (i), we have
\begin{eqnarray}\label{mmmm}
m+1=dist(z,{\bf x}_{t^{\prime}})\leq max\{dist(z,{\bf x}_{t_E}),dist(z,\mu_p({\bf x}_{t_E}),dist(z,\mu_q({\bf x}_{t_E})\},
\end{eqnarray}
 and, on the other hand,
$$max\{dist(z,{\bf x}_{t_E}),dist(z,\mu_p({\bf x}_{t_E}),dist(z,\mu_q({\bf x}_{t_E})\}\leq r+3\leq m,$$
 which contradicts to (\ref{mmmm}). So, the $p$-th components of $d$-vectors of $x_{p;t_b},x_{q;t_b},x_{q;t_r},x_{p;t_{r+1}}$, $x_{q;t_{r+1}}, x_{p;t_{r+2}}$ with respect to ${\bf x}_{t^{\prime}}$ are nonnegative.
\end{proof}

{\em Return to the proof of Theorem \ref{mainthm}.}

The proof in case (II):

 Applying Theorem \ref{Lee} (iii) at the vertex $t_{r+1}$ with respect to the directions $q$ (in case of (\ref{eqnseq1})) and $p$ (in case of (\ref{eqnseq2})), we get that either
$$z\in\mathbb{NP}[x_{p;t_a},x_{q;t_a},x_{p;t_r},x_{p;t_{r+1}},x_{q;t_{r+1}},x_{q;t_{r+2}};(x_{i;t_{r+1}}^{\pm1})_{i\neq p,q}],$$
or
$$z\in\mathbb{NP}[x_{p;t_b},x_{q;t_b},x_{q;t_r},x_{p;t_{r+1}},x_{q;t_{r+1}},x_{p;t_{r+2}};(x_{i;t_{r+1}}^{\pm1})_{i\neq p,q}].$$
Then by Lemma \ref{lempositive},  the $p$-th component of $d^{t^{\prime}}(z)$ is nonnegative and so,  $d^{t^{\prime}}(z)$ is in $\mathbb N^n$.
\end{proof}
{\bf Acknowledgements:}\; This project is supported by the National Natural Science Foundation of China (No.11671350 and No.11571173).


\end{document}